\documentclass[11 pt]{amsart}

\usepackage{titlesec}

\titleformat{\section}[runin]
{\normalfont\normalsize\bfseries}
{\thesection}
{1em}
{\addperiod}
\titleformat{\subsection}[runin]
{\normalfont\normalsize\bfseries}
{\thesubsection}
{1em}
{\addperiod}
\titleformat{\subsubsection}[runin]
{\normalfont\normalsize\bfseries}
{\thesubsubsection}
{1em}
{\addperiod}
\newcommand{\addperiod}[1]{#1.}

\usepackage[margin = 1.2 in]{geometry}
\usepackage{amsmath}
\usepackage{amsxtra}
\usepackage{amscd}
\usepackage{amsthm, hyperref, paralist, thmtools, thm-restate}
\usepackage{amsfonts}
\usepackage{amssymb}
\usepackage[all]{xy}
\usepackage{graphicx}
\usepackage{tikz-cd}
\usetikzlibrary{positioning}
\usepackage[charter]{mathdesign}
\usepackage[colorinlistoftodos, textsize=tiny]{todonotes}

\RequirePackage{color}
\definecolor{myred}{rgb}{0.75,0,0}
\definecolor{mygreen}{rgb}{0,0.5,0}
\definecolor{myblue}{rgb}{0,0,0.65}

\usepackage{tikz}
\usetikzlibrary{matrix,arrows,decorations.pathmorphing}

\theoremstyle{plain}
\newtheorem{theorem}{Theorem}[section]
\newtheorem{proposition}[theorem]{Proposition}
\newtheorem{lemma}[theorem]{Lemma}
\newtheorem{corollary}[theorem]{Corollary}
\theoremstyle{definition}
\newtheorem{definition}[theorem]{Definition}
\newtheorem{remark}[theorem]{Remark}
\newtheorem{example}[theorem]{Example}

\theoremstyle{remark}
\newtheorem{notation}[theorem]{Notation and Setup}
\numberwithin{equation}{section}
  
\newcommand\nc{\newcommand}
\nc\on{\operatorname}
\nc\renc{\renewcommand}

\newcommand\G{{\mathbf G}}
\renewcommand\P{{\mathbf P}}

\newcommand\A{{\mathbf A}}

\newcommand\fm{{\mathfrak m}}

\newcommand \m{{\mathcal M}}
\newcommand \barm{\overline{{\mathcal M}}}

\renewcommand \to{\longrightarrow}

\DeclareMathOperator\Hyp{Hyp}

\DeclareMathOperator\Spec{Spec}

\DeclareMathOperator\length{length}

\DeclareMathOperator\Span{Span}

\DeclareMathOperator\PGL{PGL}
\DeclareMathOperator\dm{dim}

\def\listtodoname{List of Todos}
\def\listoftodos{\@starttoc{tdo}\listtodoname}

\title{Moduli of linear sections of a general hypersurface}
\author{Anand Patel} \address{Dept. of
Mathematics, Boston College, Chestnut Hill, MA 02467 USA}
\email{anand.patel@bc.edu}

\date{\today}

\begin{document}

\maketitle

\section{Introduction and main results}

One way to create  new projective varieties from a given variety $X \subset \P^{r}$ is to intersect  with  linear spaces. It is only natural to want to understand the collection of varieties obtained in this way.  

There are different approaches to assigning ``moduli'' to the linear sections of $X$.  Algebraically, we may assign a point in an appropriate Hilbert scheme $H$.  This works locally as we perturb the linear section, but in order to produce a globally defined variation, we must assign the corresponding point in the quotient of the Hilbert scheme  by the action of $\PGL$. The usual technical issues with GIT emerge from this approach, but ultimately one can construct an algebraic moduli map: 
\begin{align*}
  \mu_{m}: \G\left( m,r \right) \dashrightarrow H / \PGL
\end{align*} 
which, to a general $m$-plane $\Lambda$ assigns the variety $[X \cap \Lambda]$.
Over the complex numbers, the Hodge structure on the middle cohomology of a linear section $X \cap \Lambda$  provides another way of attaching moduli to the section. Again, under ideal conditions, one gets a period map: 
\begin{align*}
  \pi_{m}: \G\left( m,r \right) \dashrightarrow {\mathcal D}/\Gamma
\end{align*}
In either case, whether algebraic or analytic, we call these maps {\sl moduli maps}.  Three questions of particular interest are: 
\begin{enumerate}
  \item What is the dimension of the image of a moduli map?
  \item When is a moduli map generically injective?
  \item If a moduli map is generically finite, what is its degree?

\end{enumerate}
These questions, especially the first two, are certainly not new, and have held the attention of many authors. The first question is infinitesimal in nature -- one has to compute the rank of the differential of the moduli map. The second question is global in nature, and the third question is enumerative.

Although the local and global analyses of moduli maps  would seem mostly disconnected, the seminal work of Donagi \cite{donagi}  shows that the infinitesimal variation of Hodge structure (IVHS) of a variety can, in many situations, recover the variety itself.  From the perspective of varying linear slices, we may recast Donagi's result as proving the generic injectivity of the moduli map for hyperplane sections of the $d$-uple Veronese varieties $X \subset \P^{r}$. In the complex analytic setting, such generic injectivity results are called {\sl weak Torelli} or {\sl generic Torelli} theorems. 

In this paper, we prove a generic Torelli theorem  for the family of $m$-plane slices of a generic hypersurface $X$. Numerous generic Torelli statements exist in the literature, including results for complete intersections in \cite{konno} and \cite{terasoma}, but the problem of generic Torelli for linear sections of a fixed hypersurface seems to have fallen through the cracks. 

Question (1) is easy to answer in our setting -- a simple first-order calculation left to the reader shows that the images of all  moduli maps are as large as possible. We note, however, that this question is easy only because of the nebulous requirement that $X$ be ``generic''.  Once we try to qualify what ``generic'' means, question (1) becomes much more challenging. For example, after considerable effort, in \cite{joemazurpanda} it is shown (among other things) that if we let $X$ be an arbitrary {\sl smooth} hypersurface of degree $d$, then, provided $r$ is very very large, the moduli maps for linear slices has image as large as possible. 

On the other extreme, an interesting result of Beauville \cite{beauville} characterizes the degree $d$ smooth  hypersurfaces $X$ for which the moduli map of hyperplane sections is constant -- this only happens when $d=2$, or when $X$ is a Fermat hypersurface and $d-1$ is a power of the characteristic of the (algebraically closed) ground field.

We  let $\Hyp(d,m)$ denote the moduli space of degree $d$ hypersurfaces in $\P^{m}$, modulo projective equivalence.

\begin{theorem}
  \label{theorem:main}
  Assume $d > 2r + 1$, and let $X \subset \P^{r}$ be a general hypersurface of degree $d$. Then the moduli maps $\mu_{m} : \G\left( m,r \right) \dashrightarrow \Hyp(d,m)$ are generically injective for all $m$.
\end{theorem}

The generic Torelli theorems mentioned earlier are proven using variants of Donagi's approach. Our technique for proving \autoref{theorem:main} is completely different; we deduce the result by a simple reduction to the case $m=1$ of lines. 

In the case of lines, we study the moduli map $\mu_{1}$  by resolving its indeterminacy (at some special points), and then computing its degree at special points in the blown up domain. The points we consider are deep in the locus of indeterminacy, so the details of the resolution require special attention.  We are able to carry out the resolution of indeterminacy because of an elementary, but fundamental {\sl versality} result (\autoref{corollary:versalkincident}) which essentially allows us to write the moduli map explicitly. This versality result is interesting in its own right, and should be true in a much larger context.  Such a generalization will be the subject of future work.

Our analysis yields complete solutions to questions (2) and (3) in the case of lines. In order to answer question (2), we study the monodromy on the set of $k$-incident lines to $X$, where $k = d-(2r-2)$.  These are lines which meet $X$ at $k$ points, i.e. {\sl as infrequently as possible}.

The set of $k$-incident lines to $X$ divides itself into groups according to the multiplicity vector $(m_{1}, \dots, m_{k})$,  $\sum_{i=1}^{k}m_{i}=d$, describing the way the lines meet $X$. Clearly, monodromy preserves the multiplicity vector of a $k$-incident line.

We prove the following generalization of a monodromy result on flex lines of plane curves found in \cite{joe:galois}:
\begin{theorem}
  \label{theorem:monodromymain}
  Let $X$ be a general hypersurface of degree $d \geq 2r$, let $k = d - (2r-2)$, and let $(m_{1}, \dots, m_{k})$ be a multiplicity vector with at least two distinct entries. 

  Then the monodromy group of the set of $k$-incident lines to $X$ with multiplicity vector $(m_{1}, \dots, m_{k})$ is the full symmetric group.
\end{theorem}

 The key to answering question (3) for the moduli map $\mu_{1}$ is to analyze the {\sl tri-incident} lines of $X$ -- these are lines meeting $X$ at exactly three points.

\begin{theorem}
  \label{theorem:main2}
  Let $X \subset \P^{r}$ be a general hypersurface of degree $d=2r+1$. Then 
  \begin{align*}
    \deg \mu_{1} = 2\sum_{a \geq b >1}^{}n_{a,b,1} + 4n_{d-2,1,1}
  \end{align*}
  where $n_{i,j,k}$ is the number of tri-incident lines to $X$ with intersection multiplicities $\left( i,j,k \right)$. 
\end{theorem}

\begin{remark}
  \label{remark:periodmap}
  \autoref{theorem:main}, along with Donagi's theorem on the generic injectivitiy of the period map $\Hyp(d,m) \dashrightarrow {\mathcal D}/\Gamma$, implies that the composite period map $\pi_{m}: \G\left( m,r \right) \dashrightarrow {\mathcal D}/\Gamma$ is injective for a generic degree $d$ hypersurface $X$ with $d > 2r+1$.  
\end{remark}

\begin{remark}
  \label{remark:notexhausted}
  \autoref{theorem:main} does not exhaust all cases where $\dm \G\left( m,r \right) < \dm \Hyp(d,m)$. We also remark that a similar argument as found in the proof of \autoref{theorem:main} shows that knowledge of generic injectivity for $m$-planes implies the same for larger planes. 

\end{remark}

\begin{remark}
  \label{remark:genericityassumption}
  Clearly, some sort of genericity condition on $X$ is required for an affirmative answer to question (2): Consider $X$ a Fermat hypersurface. 

  Interestingly, the example of the Fermat quintic in $\P^{2}$ shows that a genericity assumption stronger than ``smooth'' is required for a uniform answer to question (3).  In \cite{cadmanlaza} , the degree of $\mu_{1}$ for a general quintic is $420$ (in agreement with \autoref{theorem:main2}), while for the Fermat quintic the degree drops to $150$. 
\end{remark}

\subsection{Acknowledgements} I thank Francesco Cavazzani, Anand Deopurkar, Igor Dolgachev, Joe Harris, and Claire Voisin for useful conversations.

\subsection{Conventions} We work over an algebraically closed field of characteristic zero. If $F$ is a sheaf, we let $\Gamma (F)$ denote the space of global sections of $F$.

\section{Versality}

\subsection{Local models}
Throughout this section  $X \subset \P^{r}$ will  be a smooth hypersurface not containing any lines.

\begin{definition}
  \label{definition:tautological}
  We let $\pi : \P \to \G\left( 1,r \right)$ denote the tautological $\P^{1}$-bundle.
\end{definition}

\begin{definition}
  \label{definition:incidencecorrespondence}
  The {\sl point-line incidence correspondence for $X$} is the variety $\Sigma \subset \P$ defined by 
  \begin{align*}
    \Sigma := \left\{ (x, [\ell]) \in X \times \G\left( 1,r \right) \mid x \in \ell \right\}
  \end{align*}
  We let $\varphi, \pi$ be the projections of $\Sigma$ to the first and second factors. 
\end{definition}

\begin{lemma}
  \label{lemma:sigmasmooth}
  $\Sigma$ is a smooth variety.
\end{lemma}
\begin{proof}
  The projection $\varphi : \Sigma \to X$ is a projective bundle, with fibers isomorphic to $\P^{r-1}$.  
\end{proof}

Clearly, since $X$ is assumed to contain no lines, the map $\pi : \Sigma \to \G\left( 1,r \right)$ is a finite branched cover of degree $d$.

\begin{definition}
  \label{definition:Wn}
  We let $W_{n}$, $(n > 1)$, denote the pair $(Y_{n},\P^{1} \times \A^{n-1} \to (\A^{n-1},0))$, with $Y_{n} \subset \P^{1} \times \A^{n-1}$  defined by the equation 
  \begin{align*}
    s^{n} + a_{n-2}s^{n-2}t^{2} + a_{n-3}s^{n-3}t^{3} + \dots + a_{0}t^{n} = 0,
  \end{align*}
  where $a_{i}$ are coordinates on $\A^{n-1}$, and $[s:t]$ are homogeneous coordinates on $\P^{1}$. 
\end{definition}

\begin{remark}
  \label{remark:dehomogenize}
  To simplify notation, we will often write the affine equation $$z^{n} + a_{n-2}z^{n-2} + \dots a_{0}= 0$$ for the equation defining $W_{n}$. $W_{n}$ is often referred to as the ``mini-versal unfolding'' of the singularity $z^{n} = 0$.
\end{remark}

\begin{proposition}
  \label{proposition:versality}
  Let $B$ be a smooth variety, $b \in B$ a marked point.  Suppose $Z \subset \P^{1} \times B$ is a divisor such that $Z_{b} \subset \P^{1} \times \left\{ b \right\}$ is a scheme isomorphic to $s^{n} = 0$. Then there is an \'etale neighborhood $U$ of $b$ and a map $\mu : \left( U,b \right) \to \left( \A^{n-1},0 \right)$ such that $Z_{U} \subset \P^{1}\times U$ is isomorphic to the pullback of $W_{n}$ under $\mu$. 
\end{proposition}
\begin{proof}
  Since $B$ is smooth, it is \'etale locally isomorphic to $\A^{m}.$ So we may reduce to the case $\left( B,b \right) = \left( \A^{m},0 \right)$.

  After passing to an open set, the set $Z$ is defined by a single equation $$s^{n} + u_{n-1}s^{n-1}t + \dots + u_{0}t^{n}=0$$ where $u_{i} \in \fm_{0} \subset \Gamma (O_{\A^{m}})$. (Here $\fm_{0}$ is the maximal ideal of the origin.) We may ``complete the $n$-th power'' and find an isomorphic $Z' \subset \P^{1} \times \A^{m}$ defined by a similar equation, but with vanishing $u_{n-1}$ term. 

  The $n-1$ functions $u_{i} \in \fm_{0} \subset \Gamma \left( O_{\A^{m}} \right), i = 0, \dots n-2$ define a map $\left( \A^{m}, 0 \right) \to \left( \A^{n-1},0 \right)$ with the desired property.  
\end{proof}

\begin{proposition}
  \label{proposition:versalitymultiple}
  Let $(B.b)$ be as in the previous proposition, but now suppose $Z \subset \P^{1} \times B$ is a divisor whose restriction to $Z_{b} \subset \P^{1} \times \left\{ b \right\}$ is of the form $n_{1}p_{1} + \dots + n_{i}p_{i}$ with each $n_{j} \geq 1$.  

  Then there exists an \'etale open neighborhood $U$ of $b$ such that $Z_{U}$ is a disjoint union $ \coprod_{j}Z_{j} $, and furthermore there exists a map $(U,b) \to \left( \prod_{j}\A^{n_{j}-1},0 \right)$ such that $Z_{j}$ is isomorphic to the pullback of $W_{n_{j}}$ under the composite $(U,b) \to \left(\prod_{j}\A^{n_{j}-1},0 \right) \to \left(\A^{n_{j}-1},0\right)$.
\end{proposition}
\begin{proof}
  The argument is completely similar to the proof of \autoref{proposition:versality}.
\end{proof}

\begin{remark}
  \label{remark:nonunique}
  The maps to $W_{n}$ in these propositions are not unique, however it is easy to show that they induce a unique map on the tangent space of $B$ at $b$.
\end{remark}

\begin{definition}
  \label{definition:versal}
  Let $(B,b)$ be a smooth variety with marked point, and let $Z \subset \P^{1} \times B$ be a closed subset such that $Z_{b}$ is as in \autoref{proposition:versalitymultiple}. Then $Z \subset \P^{1} \times B \to B$ is {\sl versal at $b$} if any (and hence all) of the maps $(U,b) \to \left(\prod_{j}\A^{n_{j}-1},0 \right)$ is a local submersion, i.e. the map on tangent spaces is surjective.

\end{definition}

\begin{definition}
  \label{definition:mini-versal}
  Let $(B,b)$ be a smooth variety with marked point, and let $Z \subset \P^{1} \times B$ be a closed subset such that $Z_{b}$ is as in \autoref{proposition:versalitymultiple}. Then $Z \subset \P^{1} \times B$ is {\sl mini-versal at $b$} if any (and hence all) of the maps $(U,b) \to \left(\prod_{j}\A^{n_{j}-1},0 \right)$ is a local isomorphism.
\end{definition}

\subsection{Lines meeting $X$ infrequently}

\begin{notation} \label{notation1}
Let $\sum_{i}^{k}m_{i} = d$, $\sum_{i=1}^{j}n_{i}=d$, and let $\P^{N}$ denote the projective space of degree $d$ hypersurfaces in $\P^{r}$. We define the following three types of incidence correspondences: 
\begin{enumerate}
  \item \begin{align*}
      I_{k} & \left( m_{1}, \dots, m_{k} \right) =\\ 
      & \overline{\left\{ (X,\ell) \mid X \cap \ell = \sum_{i}m_{i}p_{i}, \, \, \text{for distinct $p_{i} \in \ell$} \right\}} \subset \P^{N} \times \G\left( 1,r \right)
\end{align*}
\item  \begin{align*}
   I'_{k} &  \left( m_{1}, \dots, m_{k} \right) =\\ 
   & \overline{\left\{ (X,\ell,p_{1}, \dots, p_{k}) \mid X \cap \ell = \sum_{i}m_{i}p_{i} \right\}} \subset \P^{N} \times \P \times_{\G\left( 1,r \right)} \dots \times_{\G\left( 1,r \right)}\P 
 \end{align*}
 \item \begin{align*}
   &  I_{k,j}''  \left( m_{1},\dots, m_{k} \mid n_{1},  \dots, n_{j}\right) =\\
   &  \overline{\left\{ (X, \ell, p_{1}, \dots, p_{k}, \ell', q_{1}, \dots, q_{j}) \mid X\cap \ell = \sum_{i}m_{i}p_{i}, \, \, \,  X \cap \ell' = \sum_{i}n_{i}q_{i} \right\}}
   \end{align*}

\end{enumerate}

 The closures are taken of the sets where the objects parametrized are all distinct (for points) and do not intersect (for lines).

\end{notation}

\subsubsection{Lemmas on correspondences}

\begin{lemma}
  \label{lemma:Ikirreducible}
  $I_{k}\left( m_{1}, \dots, m_{k} \right)$ is irreducible for all $k$.
\end{lemma}
\begin{proof}
  Consider the correspondence $I_{k}'\left( m_{1}, \dots, m_{k} \right)$.  The projection to the data $(\ell, p_{1}, \dots, p_{k})$ makes $I_{k}'$ into a projective space bundle over the parameter space of tuples $\left( \ell, p_{1}, \dots, p_{k} \right)$, and hence is irreducible.  
 $I_{k}'\left( m_{1}, \dots, m_{k} \right)$ dominates $I_{k}\left( m_{1}, \dots, m_{k} \right)$ by forgetting the data of the marked points $(p_{1}, \dots, p_{k})$. The lemma follows.
\end{proof}

\begin{lemma}
  \label{lemma:2irreducible}
  The varieties $I_{k,j}''\left( m_{1}, \dots, m_{k} \mid n_{1},\dots,n_{j} \right)$ are irreducible for all $k, j$, and all multiplicities $(m_{1}, \dots, m_{k}), ( n_{1}, \dots, n_{j} )$.
\end{lemma}
\begin{proof}
  The general fiber of the projection 
  \begin{align*}
    (X ,\ell,  p_{1}, \dots, p_{k}, \ell', q_{1}, \dots, q_{j}) \mapsto (\ell, p_{1}, \dots p_{k}, \ell', q_{1}, \dots, q_{j})
  \end{align*}
  is a linear space in $\P^{N}$ with dimension independent of the pair of non-intersecting lines $\ell, \ell'$. This follows because the restriction map 
  \begin{align*}
    \Gamma (O_{\P^{r}}(d)) \to \Gamma (O_{\ell \cup \ell'}(d)) 
  \end{align*}
  is surjective. The lemma follows at once.
\end{proof}

\begin{definition}
  \label{definition:k-incident}
  Let $X \subset \P^{r}$ be a hypersuface.  A line $\ell$ is {\sl $k$-incident} to $X$  if the set-theoretic intersection $X\cap \ell$ consists of $k$  points.
\end{definition}

\begin{proposition}
  \label{proposition:finitek-incident}
  Assume $k =d -  (2r-2)$ is positive. Then a general hypersurface $X \subset \P^{r}$ of degree $d$  has finitely many $k$-incident lines, and has no $k'$-incident lines for $k' < k$.
\end{proposition}
\begin{proof}
This is standard, and left to the reader.  
\end{proof}

\subsubsection{First-order analysis near a $k$-incident line} Let $(X,\ell) \in I_{k}(m_{1},\dots,m_{k})$ be a general point. We will study the first order deformations of $X \cap \ell$, as we vary $\ell$. 

Fix homogeneous coordinates $\left[ x_{0}, \dots , x_{r} \right]$ for $\P^{r}$ such that $\ell$ is given by $x_{0} = \dots = x_{r-2} = 0$. Let $X$ be given by the degree $d$ equation 
\begin{align*}
  F = P(x_{r-1}, x_{r}) + \sum_{j=0}^{d-2}x_{j}G_{j}  
\end{align*}
where 
\begin{align*}
  P = \prod_{i=1}^{k}\left( x_{r-1} - \alpha_{i}x_{r} \right)^{m_{i}}.
\end{align*}

\begin{definition}
  \label{definition:groupG}
  Let $G \subset \PGL_{r+1}$ be the group of matrices of the form:
  \begin{align*}
    G = \begin{bmatrix}
		1 & 0 & \dots & 0 & 0 \\
		0 & 1 & \dots & 0 & 0 \\
		\hdotsfor{5} \\
		a_{r-1,0} & a_{r-1,1} & \dots & 1 & 0 \\
		a_{r,0} & a_{r,1} & \dots & 0 & 1
    \end{bmatrix}
  \end{align*}
  where $a_{i,j}$, $(i = r-1, r), (j = 0 , \dots , r-2)$ are any constants.
\end{definition}

\begin{remark}
  \label{remark:opensetgrassman}
  The orbit $G \cdot [\ell]$ is an open subset of $\G(1,r)$ isomorphic to $\A^{2r-2}$.
\end{remark}

\begin{definition}
  \label{definition:familyG}
  Let $Z \subset \ell \times G$ be the union 
  \begin{align*}
   \bigcup_{g \in G}  \left(\left(g \cdot X\right)\cap \ell,g\right).
  \end{align*}
  
\end{definition}

\begin{lemma}
  \label{lemma:ZisSigma}
  The variety $Z \subset \ell \times G$ is isomorphic to the the variety $\Sigma \subset \P$ above the open set $G \subset \G\left( 1,r \right)$. 
\end{lemma}
\begin{proof}
  This is clear.
\end{proof}

Our goal here is to understand the first-order infinitesimal neighborhood of the the family $Z \subset \ell \times G \to G$ near the point $\left[ \ell \right] \in G$. 

\begin{definition}
  \label{definition:gj}
  We let $g_{j}$ be the polynomial in $x_{r-1},x_{r}$ which is the restriction of $G_{j}$ to $\ell$.
\end{definition}

\begin{lemma}
  \label{lemma:firstorderG}
  Let $\left\{ c_{i,j} \right\}, (i=r-1,r), (j=0,\dots,r-2)$ be the elements $\partial_{a_{i,j}} \in T_{\left[ \ell \right]}G$. 

  The first-order family $Z_{\varepsilon} \subset \ell \times T_{\left[ \ell \right]}G \to T_{\left[ \ell \right]}G$ is isomorphic to the zero locus of: 
  \begin{align*}
    P\left( x_{r-1},x_{r} \right) + \sum_{j=0}^{r-2}\left( c_{r-1,j}x_{r-1}+c_{r,j}x_{r} \right)g_{j}.
  \end{align*}

\end{lemma}
\begin{proof}
  The first-order family parametrized by the Lie algebra of the group $G$ is given by applying the infinitesimal substitutions 
  \begin{align*}
    x_{0} &\mapsto x_{0} + \varepsilon(c_{r-1,0}x_{r-1}+ c_{r,0}x_{r})\\
    x_{1} &\mapsto x_{1} + \varepsilon(c_{r-1,1}x_{r-1}+c_{r,1}x_{r})\\
    &\vdots\\
    x_{r-1} &\mapsto x_{r-1}\\
    x_{r} &\mapsto x_{r}
  \end{align*}
  to the equation $F$ defining $X$ and then restricting to $\ell$. The constants $\left\{ c_{i,j} \right\}$ vary freely, and represent coordinates on the Lie algebra of $G$.  The end result is the family given in the statement of the lemma.
\end{proof}

\begin{remark}
  \label{remark:affine}
  It will usually be more convenient to use the affine notation 
  \begin{align*}
    p(z) + \sum_{j=0}^{r-2}\left( c_{r-1,j}z+c_{r,j} \right)g_{j}(z).
  \end{align*}
  
\end{remark}

Using the infinitesimal analysis above, we will now prove some lemmas which will allow us to understand the versality of the point-line correspondence $\Sigma \to \G\left( 1,r \right)$ near a $k$-incident line $\ell$, where $k = d -(2r-2)$. 

\subsection{Some lemmas on multiplication of polynomials} \label{multiplication}

\begin{notation}\label{notation2}
  Throughout  \autoref{multiplication}, we will keep the following notation.
  \begin{enumerate}
    \item Suppose $n_{1}, \dots n_{k}$ are positive integers such that $\sum_{i=1}^{k}n_{i}=d$ and $\sum_{i=1}^{k}(n_{i}-1)=2r-2$ for some positive $r$. 

    \item Let $P_{m}$ denote the vector space of polynomials in one variable $z$ with degree $\leq m$.  

    \item Finally, let $\alpha_{1}, \dots, \alpha_{k} \in \A^{1} = \Spec k[z]$ be distinct points. 

    \end{enumerate}

\end{notation}

\subsubsection{First Lemma}

\begin{lemma}
  \label{lemma:surjection}
  There exists an $r-1$-dimensional subspace $V_{r-1} \subset P_{d-1}$ such that the multiplication-then-restriction map 
  \begin{align*}
    \rho: P_{1} \otimes V_{r-1} \to \bigoplus_{i=1}^{k}k[z]/(z-\alpha_{i})^{n_{i}-1}
  \end{align*}
  is an isomorphism.
\end{lemma}
\begin{proof}
  We will explicitly construct a vector space $V_{r-1}$ with the desired property.  
  
  Our construction is slightly different depending on the parities of the numbers $n_{i}$.  So, for the convenience of the reader, let us first assume all $n_{i}$ are odd. 

  Then for each $i$, we define the polynomial 
  \begin{align*}
    p_{i} := \prod_{j \neq i}(z-\alpha_{j})^{n_{j}-1}. 
  \end{align*}
  
  Next, for each $i$  consider the vector space 
  \begin{align*}
    A_{i} := \Span\left\{p_{i} \cdot 1,p_{i} \cdot (z-\alpha_{i})^{2}, p_{i}\cdot (z-\alpha_{i})^{4}, \dots,p_{i} \cdot (z-\alpha_{i})^{n_{i}-3} \right\}. 
  \end{align*}
 
  Then it is immediately verified that 
  \begin{align*}
    \rho(P_{1} \otimes A_{i}) = (0, \dots, k[z]/(z-\alpha_{i})^{n_{i}-1}, 0 ,\dots).
  \end{align*}
  Note that $\dim A_{i} = (n_{i}-1)/2$.  By letting $V_{r-1} = A_{1}+ \dots+ A_{k}$, the lemma follows in this case.  
 
  Now suppose $n_{1}$ is even. Since $\sum_{}^{}(n_{i}-1)$ is even, we may assume without loss of generality that $n_{2}$ is also even.
   Then we define the polynomials 
   \begin{align*}
     q_{1} = (z-\alpha_{2})^{n_{2}-2}\prod_{j \neq 1,2}(z-\alpha_{j})^{n_{j}-1}\\
     q_{2} = (z-\alpha_{1})^{n_{1}-2}\prod_{j \neq 1,2}(z-\alpha_{j})^{n_{j}-1}.
   \end{align*}
   Next set 
   \begin{align*}
   B_{1,2}:= \Span\left\{ q_{1}, q_{1}(z-\alpha_{1})^{2}, q_{1}(z-\alpha_{1})^{4}  \dots q_{1}(z-\alpha_{1})^{n_{1}-2} = q_{2}(z-\alpha_{2})^{n_{2}-2}, \dots, q_{2}(z-\alpha_{2})^{2}, q_{2} \right\}.
   \end{align*}
   Then it is straightforward to verify that 
   \begin{align*}
     \rho(P_{1} \otimes B_{1,2}) = (k[z]/(z-\alpha_{1})^{n_{1}-1},k[z]/(z-\alpha_{2})^{n_{2}-1},0,0, \dots ). 
   \end{align*}
   Note that $\dim B_{1,2} = (n_{1}+ n_{2}-2)/2$. By pairing up the remaining even $n_{i}$ and creating the analogous vector spaces $B_{i,j}$, and then for odd $n_{i}$ using the spaces $A_{i}$, one can check that 
   \begin{align*}
     V_{r-1}:= \sum_{n_{i}\text{\,\,odd}} A_{i} + B_{1,2} + \dots 
   \end{align*}
   has the desired property, i.e. that $\rho$ is an isomorphism.
\end{proof}

\subsubsection{Second Lemma}

\begin{lemma}
  \label{lemma:transversetothing}
  Maintain the setup before the previous lemma. Then there exists an $r-1$-dimensional subspace $V_{r-1} \subset P_{d-1}$ such that the multiplication-then-restriction map 
  \begin{align*}
    \rho': P_{1}\otimes V_{r-1} \to k[z]/(z-\alpha_{1})^{n_{1}} \oplus_{i=2}^{k}k[z]/(z-\alpha_{i})^{n_{i}-1}
  \end{align*}
  has the following properties: 
  \begin{enumerate}
    \item $\rho'$ is injective.
    \item The image of $\rho'$ does not contain the element 
      \begin{align*}
	w = ((z-\alpha_{1})^{n_{1}-1},0,0,\dots).
      \end{align*}
      
    \end{enumerate}
\end{lemma}
\begin{proof}
  The spaces $V_{r-1}$ constructed in the proof of \autoref{lemma:surjection} satisfy these two properties.  We leave the verification to the reader.
\end{proof}

\begin{remark}
  \label{remark:openproperty}
  The conditions in these two lemmas are open -- a general $V_{r-1} \subset P_{d-1}$ will have the desired properties in each lemma. 
\end{remark}

\begin{theorem}
  \label{theorem:versal}
  Let $\sum_{i=1}^{k}m_{k}=d$, and suppose $\sum_{i}^{}m_{i}-1 = 2r-2$.  Suppose $\left( X,\ell \right)$ is a general point in $I_{k}\left( m_{1}, \dots , m_{k} \right)$.  Then the point-line correspondence $\pi: \Sigma \to \G\left( 1,r \right)$ is mini-versal near the $k$-incident line  $[\ell] \in \G\left( 1,r \right)$. 
\end{theorem}
\begin{proof}
  The map $\rho$ in \autoref{lemma:surjection} is the differential of (any of) the induced maps  
  \begin{align*}
    \left( U, \left[ \ell \right] \right) \to W_{m_{1}} \times \dots W_{m_{k}}
  \end{align*}
  from \autoref{proposition:versalitymultiple} (Here $U$ is an \'etale neighborhood of $\ell$). The theorem follows immediately from \autoref{proposition:versalitymultiple}.  
\end{proof}

\begin{corollary}
  \label{corollary:versalkincident}
  Let $X \subset \P^{r}$ be a general degree $d \geq 2r-1$ hypersurface, and let $k = d-(2r-2)$. Then 
  \begin{align*}
    \left( \Sigma \subset \P \to \G\left( 1,r \right) \right)
  \end{align*}
  is versal at every $k$-incident line $\left[ \ell \right] \in \G\left( 1,r \right)$.
\end{corollary}
\begin{proof}
  
  When $k = d-2(r-1)$, the irreducible variety $I_{k}\left( m_{1}, \dots, m_{k} \right)$ maps generically finitely onto $\P^{N}$. The locus consisting of pairs $(X,\ell) \in I_{k}(m_{1}, \dots , m_{k})$ for which the point-line correspondence fails to be versal at $\ell$ is a strict (closed) subset of $I_{k}\left( m_{1}, \dots ,m_{k} \right)$ by \autoref{theorem:versal}. Since this is true for all multiplicity data $\left( m_{1}, \dots, m_{k} \right)$, we conclude that there is a Zariski open subset of $\P^{N}$ for which the corollary holds.
\end{proof}

\section{Maximal monodromy} The purpose of this section is to establish  a basic result about the monodromy of lines meeting a general hypersurface at the least possible number of points.   

\begin{theorem}
  \label{theorem:maximalmonodromytheorem}
  Suppose  $k = d - (2r-2) \geq 2 $, and suppose $\sum_{i=1}^{k}m_{i}=d$ is a partition with at least two distinct parts.
    Then the Galois group of the generically finite projection 
  \begin{align*}
    \varphi : I_{k}\left( m_{1}, \dots, m_{k} \right) \to 
    \P^{N}  \end{align*}
  is the full symmetric group.
\end{theorem}

Our approach to proving \autoref{theorem:maximalmonodromytheorem} is standard. Let $G$ denote the Galois group of $\varphi$. We break the proof up into two standard steps, in increasing difficulty: \begin{enumerate}
  \item Show $G$ is $2$-transitive.
  \item Show $G$ contains a simple transposition.
\end{enumerate}

These two  properties imply that $G$ is the full symmetric group. Step $1$ is easy and quite standard, whereas Step $2$ requires some care, and will occupy most of our attention.

\subsection{$2$-transitivity}

\begin{lemma}
  \label{corollary:2transitive}
  The monodromy group  $G$ is $2$-transitive.
\end{lemma}
\begin{proof}
  It suffices to show that the variety $$I_{k}(m_{1}, \dots, m_{k}) \times_{\P^{N}} I_{k}\left( m_{1}, \dots, m_{k} \right) \setminus \Delta$$ is irreducible. (Here $\Delta$ is the diagonal.) 
  By \autoref{lemma:2irreducible} we know that $I_{k,k}''\left( m_{1}, \dots, m_{k} \mid m_{1}, \dots, m_{k} \right)$ is irreducible. Since this correspondence dominates $I_{k}(m_{1}, \dots, m_{k}) \times_{\P^{N}} I_{k}\left( m_{1}, \dots, m_{k} \right) \setminus \Delta$, we conclude the  claimed irreducibility, and hence the corollary. 
\end{proof}

\begin{remark}
  \label{remark:incidentlines}
  We do not have to concern ourselves with incident pairs of lines $\ell, \ell'$ -- the reader can check that a general $X \in \P^{N}$ will not have an intersecting pair of $k$-incident lines by an elementary dimension count. 
\end{remark}

\subsection{In pursuit of a simple transposition} 

\begin{notation}\label{notation3}
Until otherwise specified, we adopt the following notation: 

  \begin{enumerate}
    \item   $(X, \ell) \in I_{k-1}\left( m_{1}+m_{2}, m_{3}, \dots, m_{k} \right)$ is a general point, with $k = d-(2r-2)\ \geq 2$. Furthermore, we assume $m_{1} \neq m_{2}$.

    \item We let $\A^{m_{1} + m_{2}- 1}$ denote the base of the mini-versal deformation $W_{m_{1}+m_{2}}$. 

    \item We let $Z_{m_{1},m_{2}} \subset \A^{m_{1}+m_{2}-1}$ be the locus parametrizing divisors of the form $m_{1}p + m_{2}q$. 

    \item $\pi : \Sigma \to \G\left( 1,r \right)$ denotes the point-line correspondence for $X$.
    \end{enumerate}
\end{notation}

\begin{lemma}
  \label{lemma:Zcurve}
  $Z_{m_{1},m_{2}}$ is a uni-branched curve with a multiplicity two singularity at the origin $0 \in \A^{m_{1}+m_{2}-1}$. 
\end{lemma}
\begin{proof}
  The family of deformations of $z^{m_{1}+m_{2}} = 0$ parametrized by the set $Z_{m_{1},m_{2}}$ is the image of the map
  \begin{align*}
    t \mapsto (z-m_{2}t)^{m_{1}}(z+m_{1}t)^{m_{2}}.
  \end{align*}
  By expanding this  polynomial in terms of $z$, we obtain an expression
  \begin{align*}
    z^{m_{1}+m_{2}} + c_{2}t^{2}z^{m_{1}+m_{2}-2} + c_{3}t^{3}z^{m_{1}+m_{2}-3} + \dots + c_{m_{1}+m_{2}}t^{m_{1}+m_{2}} 
  \end{align*}
  where, under the assumption that $m_{1} \neq m_{2}$, all coefficients $c_{j}$ are nonzero. The lemma follows. 
\end{proof}

\begin{corollary}
  \label{corollary:multiplicitytwointersect}
  Let $D \subset \A^{m_{1}+m_{2}-1}$ be a divisor given by the vanishing of an equation $f(a_{2}, \dots, a_{m_{1}+m_{2}})=0$, where $f$ is such that $0 \in D$, and such that $\frac{\partial}{\partial a_{2}}f(0) \neq 0$. Further suppose $D$ does not contain $Z_{m_{1},m_{2}}$. 

  Then $$\length (D \cap Z_{m_{1}, m_{2}}) = 2.$$ 
\end{corollary}
\begin{proof}
  This easily follows from the parametrization of $Z_{m_{1},m_{2}}$ found in the  proof of \autoref{lemma:Zcurve}.
\end{proof}

\begin{proposition}
  \label{proposition:transversefamily}
  Let $(X,\ell) \in I_{k-1}(m_{1}+m_{2}, m_{3}, \dots, m_{k})$ be a general point, and let $\pi: \Sigma \to \G\left( 1,r \right)$ be the point-line correspondence. 

  Denote by $$\nu: (U, [\ell]) \to (\A^{m_{1}+m_{2}-1} \times \A^{m_{3}-1} \times \dots \A^{m_{k}-1},0)$$ a local map to the mini-versal deformation space where $U \subset \G\left( 1,r \right)$ is a sufficiently small \'etale neighborhood of $\left[ \ell \right]$. Then
  \begin{enumerate}
    \item The map $\nu$ is an immersion near $[\ell]$. 
    \item The image $\nu(U)$ is transverse to $Z_{m_{1},m_{2}} \subset (\A^{m_{1}+m_{2}-1}, 0, \dots, 0)$ in the sense of \autoref{corollary:multiplicitytwointersect}.
  \end{enumerate}
\end{proposition}
\begin{proof}
  The two statements are proved at the level of tangent spaces.  They follow, respectively, by applying points $1$ and $2$ in \autoref{lemma:transversetothing}.
\end{proof}

\subsubsection{``No unintended consequences'' lemma} Now suppose 
\begin{align*}
  (X,\ell,\ell') \in I_{k-1,k}(m_{1}+m_{2}, m_{3}, \dots, m_{k} \mid m_{1}, \dots, m_{k})
\end{align*}
 is a general point. The next lemma and corollary show that the condition of possessing the $k-1$-incident line $\ell$ has no effect on the $k$-incident line $\ell'$ of $X$. Hence, there are in general ``no unintended consequences elsewhere'' for having a $k-1$-incident line. 

 \begin{lemma}
   \label{lemma:nocommunication}
   The point-line correspondence 
   \begin{align*}
     \pi : \Sigma \to \G\left( 1,r \right)
   \end{align*}
   is mini-versal near $[\ell'] \in \G\left( 1,r \right)$. 
 \end{lemma}
 \begin{proof}
   After using the $\PGL_{r+1}$-action, we may assume 
   \begin{align*}
     \ell = [x_{0}:x_{1}]\\
     \ell' = [x_{2}:x_{3}]
   \end{align*}
   and therefore, that the homogeneous ideal of the union $\ell \cup \ell'$ is 
   \begin{align*}
     I_{\ell \cup \ell'} = (x_{0}x_{2}, x_{1}x_{2}, x_{0}x_{3}, x_{1}x_{3}, x_{4}, x_{5}, \dots x_{r}).
   \end{align*}
   
   Therefore, we may write the defining equation of $X$ in the form 
   \begin{align*}
     F = p(x_{0},x_{1}) + q(x_{2},x_{3}) + \sum_{i=0,1 ; j=2,3}^{}x_{i}x_{j}A_{i,j} + \sum_{i=4}^{r}x_{i}G_{i}
   \end{align*}
   where $A_{i,j} \in \Gamma (O_{\P^{r}}(d-2))$, and $G_{i} \in \Gamma (O_{\P^{r}}(d-1))$. 

   We let $a_{i,j} = A_{i,j}|_{\ell'}$, and $g_{i} = G_{i}|_{\ell'}$.

   As in the proof of \autoref{lemma:firstorderG}, we apply the infinitesimal substitutions 
\begin{align*}
  x_{0}& \mapsto x_{0}+\varepsilon \ell_{0}\\
  x_{1}& \mapsto x_{1} + \varepsilon \ell_{1}\\
  x_{2}& \mapsto x_{2}\\ 
  x_{3}& \mapsto x_{3}\\
  x_{4}& \mapsto x_{4} + \varepsilon \ell_{4}\\
  \vdots& \\
  x_{r} &\mapsto x_{r} + \varepsilon \ell_{r}
\end{align*}
and then mod out by the ideal $I_{\ell'} = (x_{0}, x_{1}, x_{4}, \dots, x_{r})$. (Here $\ell_{i}$ are linear forms in the variables/$x_{2}$ and $x_{3}$.) 

The result is the infinitesimal family 
\begin{align*}
  f_{\varepsilon} = q(x_{2},x_{3}) + \varepsilon \ell_{0}(x_{2}a_{0,2} + x_{3}a_{0,3}) + \varepsilon \ell_{1}(x_{2}a_{1,2} + x_{3}a_{1,3}) + \sum_{i = 4}^{r}\varepsilon \ell_{i} g_{i}.
\end{align*}

Set $g_{0} := x_{2}a_{0,2}+x_{3}a_{0,3}$ and $g_{1} := x_{2}a_{1,2} + x_{3}a_{1,3}$. Then, general choices of the four polynomials  $a_{i,j}$ yield general choices of $g_{0}$ and $g_{1}$. 

These observations tell us that the first-order analysis of $\pi: \Sigma \to \G\left( 1,r \right)$ near $[\ell']$ is exactly the same as in the case for a general point $(X,\ell') \in I_{k}(m_{1}, \dots m_{k})$.  In particular, since $g_{0}, g_{1}, g_{4}, \dots, g_{r}$ are taken to be generic, $\pi$ is mini-versal near $[\ell']$. 

 \end{proof}
 
 \begin{corollary}
   \label{corollary:nostrangecomm}
   Let $(X,\ell) \in I_{k-1}(m_{1}+m_{2}, m_{3}, \dots, m_{k})$ be a general point. Then the point-line correspondence 
   \begin{align*}
     \pi : \Sigma \to \G\left( 1,r \right)
   \end{align*}
   is mini-versal at all $k$-incident lines $[\ell'] \in \G\left( 1,r \right)$.
 \end{corollary}
 \begin{proof}
   Indeed, the variety $I''_{k-1, k}(m_{1}+m_{2}, m_{3}, \dots, m_{k} \mid m_{1}, m_{2}, \dots, m_{k})$ is irreducible, and generically finite over $I_{k-1}(m_{1}+m_{2}, m_{3}, \dots, m_{k})$. The previous lemma shows that there is a Zariski open subset of $I_{k-1,k}$ with $\ell'$ versal. The corollary follows.
 \end{proof}

 \begin{lemma}
   \label{lemma:onlyonespecialline}
   Let $(X,\ell) \in I_{k-1}(n_{1}, \dots, n_{k-1})$ be a general point. Then $X$ does not possess any other $k-1$-incident lines.
 \end{lemma}
 \begin{proof}
   This follows from a simple dimension count -- we leave it to the reader.
 \end{proof}

 \begin{theorem}
   \label{proposition:simpletransposition}
   The monodromy group $G$ contains a simple transposition.
 \end{theorem}
 \begin{proof}
   
   Let $(X, \ell) \in I_{k-1}(m_{1}+m_{2}, \dots, m_{k})$ be a general point. 

   All $k$-incident lines of multiplicity $(m_{1}, \dots, m_{k})$ are simple, i.e. occur with multiplicity one, according to \autoref{corollary:nostrangecomm}.  The line $[\ell]$ occurs as a multiplicity $2$ $k$-incident line of type $(m_{1}, \dots, m_{k})$, according to \autoref{corollary:multiplicitytwointersect}.  

   Therefore, above the point $[X] \in \P^{N}$, the correspondence $I_{k}(m_{1}, \dots, m_{k})$ has exactly one nonreduced point of length $2$, and all other points are reduced.  

   Since the curve $Z_{m_{1},m_{2}} \subset \A^{m_{1}+m_{2}-1}$ is irreducible, the two nearby $k$-incident lines (for a nearby hypersurface $X'$) limiting to $[\ell]$ experience monodromy.

   Hence we obtain a simple transposition in the monodromy group $G$. 
 \end{proof}

 \subsubsection{Moduli of lines meeting $X$ infrequently}

\begin{definition}
  \label{definition:cross-ratios-map}
  Let $k \geq 4$. We define $i_{(m_{1}, \dots, m_{k})}: I_{k}(m_{1}, \dots , m_{k}) \dashrightarrow \m_{0,k}/S_{k}$ to be the  map sending $(X,\ell)$ to $\left[ \ell \cap_{set} X \subset \ell \right]$. Here , subscript ``set'' means ``set theoretic''.
\end{definition}

\begin{corollary}
  \label{corollary:distinct-cross-ratios}
  Assume $k = d- (2r-2) \geq 4$, and let $X$ be a general degree $d$ hypersurface. Then any two distinct $(X,\ell), (X,\ell') \in I_{k}(m_{1}, \dots, m_{k})$ have distinct images under $i_{(m_{1},\dots,m_{k})}$. 
\end{corollary}
\begin{proof}
  Let $J \subset I_{k}(m_{1}, \dots ,m_{k}) \times_{\P^{N}}I_{k}(m_{1}, \dots ,m_{k}) $ denote the closure of the locus of pairs $\left( (X,\ell), (X,\ell') \right)$ having equal images under $i_{(m_{1},\dots,m_{k})}$, and such that $\ell \cap \ell'$ is empty. 

  By \autoref{corollary:2transitive}, The fiber product $I_{k}(m_{1}, \dots ,m_{k}) \times_{\P^{N}}I_{k}(m_{1}, \dots ,m_{k}) \setminus \Delta $ is irreducible, where $\Delta$ denotes the diagonal.

  It is easy to exhibit two lines $\ell, \ell' \subset \P^{r}$ and a degree $d$ hypersurface $X$ such that $\ell$ and $\ell'$ are $k$-incident with multiplicity $(m_{1}, \dots, m_{k})$ and such that $i_{(m_{1}, \dots, m_{k})}(\ell) \neq i_{(m_{1}, \dots, m_{k})}(\ell')$.   
  
  Therefore, $J$, being a closed set, must have dimension strictly smaller than this fiber product's dimension.  But the latter is the dimension of $\P^{N}$, i.e. the locus $J$ projects to a strict subset of $\P^{N}$. This proves the corollary.  
\end{proof}

\section{Extending $\mu_{1}$ across a $k$-incident line}
\begin{notation}\label{notation4}
Throughout this section: 
\begin{enumerate}
  \item $X \subset \P^{r}$ is a general hypersurface of degree $d$.
  \item We set $k := d - (2r-2)$, and assume $k \geq 1$.
  \item We let $\pi : \Sigma \to \G\left( 1,r \right)$ be the point-line correspondence for $X$.
  \item We continue to let 
    \begin{align*}
      \mu_{1}: \G\left( 1,r \right) \dashrightarrow \Hyp(d,1)
    \end{align*}
    denote the moduli map.
  \item Finally, we let $\ell$ be a $k$-incident line of $X$ with multiplicity $\left( m_{1}, \dots, m_{k} \right)$. 
  \end{enumerate}
\end{notation}

\begin{remark}
  \label{remark:assumefamiliarity}
  We will assume that the reader has some familiarity with the moduli space $\barm_{0,d}$. We will review some important background as appropriate. 
\end{remark}

Our goal is to understand the resolution of indeterminacy of $\mu_{1}$ near $\ell$.

The moduli space $\Hyp(d,1)$ is none other than the space $\m_{0,d}/S_{d}$. Unfortunately the latter moduli space is not fine, so we will first change settings to circumvent this technical inconvenience.

\begin{definition}
  \label{definition:orderedSigma}
  We define 
  \begin{align*}
    \Sigma^{\left[ d \right]} := \overline{\Sigma \times_{\pi} \dots \times_{\pi}\Sigma \setminus \text{Diagonals}}
  \end{align*}
  We let $s: \Sigma^{[d]} \to \G\left( 1,r \right)$ denote the structural degree $d!$  map. 
\end{definition}

\begin{proposition}
  \label{proposition:smoothSigmad}
  The scheme $\Sigma^{[d]}$ is smooth at any point $\hat{[\ell]} \in s^{-1}([\ell])$. 
\end{proposition}
\begin{proof}
  This is a consequence of \autoref{corollary:versalkincident}. Indeed, \'etale locally around every $[\ell] \in \G\left( 1,r \right)$, we may assume $\pi : \Sigma \to \G\left( 1,r \right)$ is isomorphic to products of maps of the form 
  \begin{align*}
    \left\{ (z,a_{0}, \dots ,a_{n-2}) \mid z^{n} + a_{n-2}z^{n-2} + \dots a_{0}=0 \right\} \to (a_{0}, \dots ,a_{n-2})
  \end{align*}
  If $r_{1}, \dots r_{n}$ denote the roots of $z^{n} + a_{n-2}z^{n-2} + \dots  + a_{0}$, then the variety $$\Spec k\left[ r_{1}, \dots r_{n}\right]/(r_{1}+\dots+ r_{n}=0)$$ maps to $\Spec k\left[ a_{n-2}, \dots , a_{0}, \right]$ by the elementary symmetric expressions in the $r_{i}$. 

  This, in turn, means that the scheme $\Sigma^{\left[ d \right]}$, above the point $\left[ \ell \right] \in \G\left( 1,r \right)$, is \'etale locally isomorphic to a product of affine spaces $\Spec k[r_{1}, \dots, r_{n}]/(r_{1} + \dots + r_{n})$, and in particular smooth.
\end{proof}

\begin{lemma}
  \label{lemma:numberpreimages}
  The set $s^{-1}([\ell])$ has $\frac{d!}{m_{1}!m_{2}! \dots m_{k}!}$ elements.
\end{lemma}
\begin{proof}
  The group $S_{d}$ acts transitively on the set $s^{-1}([\ell])$.  From the local description of $s$ over $[\ell]$ explained in the proof of the previous proposition, we see that the stabilizer of a point in $s^{-1}([\ell])$ is a copy of $S_{m_{1}} \times \dots \times S_{m_{k}} \subset S_{d}$.  
\end{proof}

\begin{definition}
  \label{definition:Sigmasmoothd}
  We let $\pi^{\left[ d \right]} : \P \to \Sigma^{\left[ d \right]}$ denote the tautological $\P^{1}$-bundle, pulled back from $\G\left( 1,r \right)$.  $\pi^{\left[ d \right]}$ has $d$ tautological sections $\sigma_{i}$, $i = 1, \dots d$. 

  We let $U \subset \Sigma^{\left[ d \right]}$ denote the open set over which none of the sections $\sigma_{i}$ intersect one another.
  We obtain an induced moduli map which we write as: 
  \begin{align*}
    \mu_{1}^{\left[ d \right]} : U \to \barm_{0,d}.
  \end{align*}

  Finally, we let $\hat{[\ell]}$ be any point in $\Sigma^{[d]}$ lying over $[\ell]$. 
\end{definition}

\subsection{Resolution near $\hat{[\ell]}$}
Now that we have altered our setting, we begin the study of the resolution of indeterminacy of $\mu_{1}^{\left[ d \right]}$, viewed as a rational map 
\begin{align*}
  \mu_{1}^{\left[ d \right]}: \Sigma^{\left[ d \right]} \dashrightarrow \barm_{0,d}
\end{align*}
near the point $\hat{[\ell]}$.

\begin{definition}
  \label{definition:resolutionZ}
  We define $$Z \subset \Sigma^{\left[ d \right]} \times \barm_{0,d}$$ to be the closure of the graph of $m_{1}^{\left[ d \right]}$. Furthermore, we let $\alpha : Z \to \barm_{0,d}$ denote the projection to the second factor and we let $\beta: Z \to \Sigma^{\left[ d \right]}$ be the natural (birational) projection. 
\end{definition}

\subsection{Blowup description of $\barm_{0,d}$} Recall that the moduli space $\barm_{0,d}$ can be obtained as a blow up
$$\kappa: \barm_{0,d} \to (\P^{1})^{d-3}$$ which, to a general point $(\P^{1}, s_{1}, \dots, s_{d}) \in \barm_{0,d}$ assigns the coordinates of $(s_{4}, \dots, s_{d})$ once we use the $\PGL_{2}$-action to send $s_{1}$ to $0$, $s_{2}$ to $1$ and $s_{3}$ to $\infty$. 

\begin{notation}
  \label{notation:simplex} Moving forward, we will use the following objects and notation:
  \begin{enumerate}
    \item  Let $n \geq 1$ be any integer. We let $\Delta_{n}$ denote the scheme $$\Spec k\left[ r_{1}, \dots, r_{n} \right] / (r_{1}+ \dots +r_{n} = 0).$$

    \item We fix $\alpha_{1}, \dots, \alpha_{k} \in \A^{1}$, and corresponding multiplicities $m_{1}, \dots, m_{k}$ such that $\sum_{i=1}^{k}m_{i} = d$.
    \item We assume $k \geq 3$.

    \item  We define the {\sl tautological family} $F_{n}$  over $\Delta_{n}$ anchored at the point $\alpha_{i}$ to be the scheme
      \begin{align*}
	\prod_{j=1}^{n}\left( z - \alpha_{i} - r_{j} \right) = 0.
      \end{align*}
      Note: $F_{n}$ is a closed subscheme of $\Delta_{n} \times \A^{1}$.
    \item Let $\varphi_{n}(\alpha_{i}): \Delta_{n} \to (\A^{1})^{n}$ be defined by the formula: 
      \begin{align*}
	(r_{1}, \dots, r_{n}) \mapsto (\alpha_{i}+r_{1}, \dots, \alpha_{i}+r_{n}).
      \end{align*}
    \item We let 
      \begin{align*}
	\varphi_{m_{1}, \dots, m_{k}}\left( \alpha_{1}, \dots, \alpha_{k} \right): \Delta_{m_{1}} \times \dots \times \Delta_{m_{k}} \to (\A^{1})^{d} \subset (\P^{1})^{d}
      \end{align*}
      denote the map $\varphi_{m_{1}}(\alpha_{1}) \times \dots \times \varphi_{m_{k}}(\alpha_{k})$. 
    \item Let $x \in (\P^{1})^{m}$ be any point. We let $D_{x}$ denote the vector space of first order deformations of $x$ induced by the diagonal $\PGL_{2}$-action on $(\P^{1})^{m}$.

\end{enumerate}
\end{notation}

\begin{remark}
  \label{remark:ofcourse}
  Of course, $\Delta_{n} \simeq \A^{n-1}$, but we want to emphasize the particular presentation of the coordinate ring.
\end{remark}

\begin{proposition}
  \label{proposition:localisomorphism}
 Under the assumptions above,  the map 
  \begin{align*}
    \varphi_{m_{1}, \dots, m_{k}}\left( \alpha_{1}, \dots, \alpha_{k} \right): \Delta_{m_{1}} \times \dots \Delta_{m_{k}} \to (\P^{1})^{d}
  \end{align*}
 has the following properties 
 \begin{enumerate}
   \item The induced map on tangent spaces is injective at $0 \in \Delta_{m_{1}} \times \dots \times \Delta_{m_{k}}$.
   \item The tangent space $T_{0}\Delta_{m_{1}} \times \dots \Delta_{m_{k}}$ and the space $D_{\varphi_{m_{1}, \dots, m_{k}}(\alpha_{1}, \dots, \alpha_{k})(0)}$ are linearly independent vector subspaces of the tangent space of $(\P^{1})^{d}$ at  $\varphi_{m_{1}, \dots, m_{k}}(\alpha_{1}, \dots, \alpha_{k})(0)$.
   \end{enumerate}
\end{proposition}
\begin{proof}
  The first point is basically by definitions. 

  An element $v \in D_{\varphi_{m_{1}, \dots, m_{k}}(\alpha_{1}, \dots, \alpha_{k})(0)} $ is determined by its effect at the three distinct points $\alpha_{1}, \alpha_{2}$ and $\alpha_{3}$ in $\P^{1}$.  Assume the first order deformation $v$ has the following effect: 
  \begin{align*}
    \alpha_{1} \mapsto \alpha_{1} + \varepsilon v_{1}\\
    \alpha_{2} \mapsto \alpha_{2} + \varepsilon v_{2}\\
    \alpha_{3} \mapsto \alpha_{3} + \varepsilon v_{3}.
  \end{align*}

  Due to the defining equation $(r_{1}+ \dots + r_{n} = 0)$ for $\Delta_{n}$, In order for $v$ to be in $T_{0}\Delta_{m_{1}} \times \dots \times \Delta_{m_{k}}$, we must have: 
  \begin{align*}
    m_{1}v_{1} = 0\\
    m_{2}v_{2} = 0\\
    m_{3}v_{3} = 0.
  \end{align*}
  But then all $v_{i}$ are zero, and hence $v$ is.  This is what we wanted to show. 
\end{proof}

\begin{corollary}
  \label{corollary:localisomorphismontoimage}
  The induced map 
  \begin{align*}
    \bar{\varphi}: \Delta_{m_{1}} \times \dots \times \Delta_{m_{k}} \to (\P^{1})^{d} / \PGL_{2} \simeq (\P^{1})^{d-3}
  \end{align*}
  is a local immersion near the point $0 \in \Delta_{m_{1}} \times \dots \times \Delta_{m_{k}}$.
\end{corollary}
\begin{proof}
  This follows from \autoref{proposition:localisomorphism}.
\end{proof}

\begin{corollary}
  \label{corollary:mainresolutioncorollary}
  There exists an \'etale neighborhood $V \subset \Sigma^{[d]}$ containing $\hat{[\ell]}$ such that 
  \begin{align*}
    \alpha: \beta^{-1}(V) \to \barm_{0,d}
  \end{align*}
  is an isomorphism onto its image. In particular, $\alpha$ is unramified at all points of $\beta^{-1}(\hat{[\ell]})$.

\end{corollary}
\begin{proof}
  This follows from \autoref{corollary:localisomorphismontoimage}, from the fact that $\Sigma^{[d]}$ is \'etale locally isomorphic to $\Delta_{m_{1}} \times \dots \times \Delta_{m_{k}}$ near $\hat{[\ell]}$, and the description of $\barm_{0,d}$ as a blowup $\kappa: \barm_{0,d} \to (\P^{1})^{d-3}$. 
\end{proof}

\subsection{Stable reduction} For the reader's convenience, we briefly review the procedure of stable reduction for families of $d$-pointed rational curves. 

\begin{notation} \label{notation5} The following notation will appear frequently in this section:

  \begin{enumerate}
    \item Let $B$ be a smooth curve, and $b \in B$ a specified point. 
      
    \item Let $\sigma_{1}, \dots, \sigma_{d}$ be distinct sections of the projection $\P^{1} \times B \to B$.  We assume that away from $b$, the sections $\sigma_{i}$ are mutually disjoint. We let $\P^{1}_{b}$ denote $\P^{1} \times \left\{ b \right\}$. 

    \item Let 
\begin{align*}
  D := \left( \bigcup_{i}\sigma_{i} \right) \cap \P^{1}_{b}.
\end{align*}
 and write $D = p_{1} + \dots p_{n} + a_{1}q_{1} + \dots a_{k-n}q_{k-n}$, with $a_{i}>1$ and all points $p_{j}, q_{j}$ distinct.

 \item We denote stable genus $0$ curves in the usual way as $(P,s_{1}, \dots, s_{d})$. The nodes and marked points of $(P,s_{1}, \dots, s_{d})$ are called {\sl special} points. 
\end{enumerate}
\end{notation}

\begin{proposition}
  \label{proposition:stablereduction}
  Assume $k \geq 3$. Then the stable replacement $\left( P, s_{1}, \dots , s_{d} \right) \in \barm_{0,d}$ of the pair $D \subset \P^{1}_{b}$ is a union of
  \begin{enumerate}
    \item $\P^{1}_{b}$ and
    \item $T_{i}, i = 1, \dots k-n$, possibly-nodal rational tails with $T_{i} \cap \P^{1}_{b} = q_{i}.$ 
  \end{enumerate}
  The marked points $\left\{ s_{1}, \dots, s_{d} \right\}$ are distributed as follows: on $\P^{1}_{b}$ the points $p_{1}, \dots, p_{n}$ are marked, while each $T_{i}$ contains $a_{i}$ marked points. 
\end{proposition}
\begin{proof}
  We recall the stable reduction procedure, from which the proposition is immediately verified. 

  \begin{enumerate}
    \item We view $C := \bigcup \sigma_{i} \subset \P^{1} \times B$ as a curve on a surface, with singularities $q_{i} \in \P^{1}_{b}$.
    \item We resolve the singularities $q_{i} \in C$ by blowing up repeatedly until all branches of $C$ through $q_{i}$ are separated.
    \item Since each blow up in the resolution process happens at smooth centers, the new fiber $P$  of the blown up surface over the point $b \in B$ is the union of the original $\P^{1}_{b}$ and several rational tails attached to $\P^{1}_{b}$ at the points $q_{i}$. 
    \item We contract any components of $P$ which have $\leq 2$ special points.

\end{enumerate}

  The proposition follows from this description of the stable reduction process.
\end{proof}

\begin{remark}
  \label{remark:stablereductionprocess}
  The process of stable reduction described in the proof of \autoref{proposition:stablereduction} will be referred to frequently in the remainder of this section.  We emphasize that when $k \geq 3$, the original fiber $\P^{1}_{b}$ {\sl naturally} persists as a component of the stable reduction $P$.
\end{remark}

\begin{definition}
  \label{definition:chain}
  A {\sl chain} is a nodal curve $T = P_{1} \cup \dots \cup P_{j}$ where each $P_{i} \equiv \P^{1}$, and $P_{i}$ intersects $P_{i+1}$ at one node, for all $i = 1, \dots j-1$, with no further intersections between components. 
\end{definition}

\begin{proposition}
  \label{proposition:limitedorigin}
  Assume $(P,s_{1}, \dots , s_{d}) \in \barm_{0,d}$ is a stable curve which is the nodal union of: 
  \begin{enumerate}
    \item a $\P^{1}$ with $n$ marked points $p_{1}, \dots, p_{n}$, and
    \item $k-n$ chains $T_{i}$, $i = 1, \dots k-n$, with each $T_{i}$ intersecting $\P^{1}$ once at a point $q_{i}$, every component of $T_{i}$ having exactly three special points, and each $T_{i}$ containing $a_{i}>1$ marked points.  
  \end{enumerate}
  Suppose $(P, s_{1}, \dots, s_{d})$ is the stable replacement for $D \subset \P^{1}_{b}$ in \autoref{notation5}.  

  Then 
  \begin{enumerate}
    \item $D \subset \P^{1}_{b}$ is supported on $\leq k$ points. 
    \item If $D$ is supported on $k$ points, and if we assume  $k \geq 4$,  then $D$ must equal $p_{1}+ \dots + p_{n} + a_{1}q_{1} + \dots a_{k-n}q_{k-n}$.
    \end{enumerate}
\end{proposition}
\begin{proof}
  These two points follow from the stable reduction procedure described in the proof of \autoref{proposition:stablereduction}, and from the observations in \autoref{remark:stablereductionprocess}.
\end{proof}

\subsection{Dandelions and the proof of \autoref{theorem:main}} 
\begin{definition}
  \label{definition:k-dandelion}
  Any point $[(P,s_{1},\dots,s_{d})] \in \barm_{0,d}$ having dual graph  
  \tikzset{main node/.style={thick,circle,draw},}
  
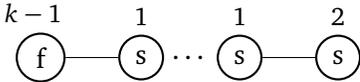
  \begin{tikzpicture}
    \node[main node, label={[xshift=-.1cm]\small $k-1$}] (1) {f};
    \node[main node, label={\small $1$}] (2) [right = .7cm of 1] {s};
    \node[main node, label={\small $1$}] (3) [right = .7cm of 2] {s};
    \node[main node, label={\small $2$}] (4) [right = .7cm of 3] {s};

    \path (1) edge (2);
    \path (2) -- node[auto=false]{\ldots} (3);
    \path (3) edge (4);
  \end{tikzpicture}
  
  is called a {\sl $k$-dandelion}. There are a total of $d-k$ vertices labelled ``s''.
(Here, the numbers above a vertex indicate the number of marked points on the corresponding component of the stable curve.)
  The {\sl stem} of the dandelion is the curve which is the union of components labelled ``s''. The {\sl flower} of the $k$-dandelion is the component corresponding to the vertex labelled ``f''.
\end{definition}

\begin{proposition}
  \label{proposition:wheredandelions}
  Let $\left[P  \right] \in \alpha(Z)$ be a $k$-dandelion with $k = d-(2r-2)\geq 4$. Then the the map $\alpha$ is unramified at every preimage of $\left[ P \right]$. Furthermore, every preimage of $\left[ P \right]$ projects, under $\beta$ to a point in $\Sigma^{[d]}$ lying over a $k$-incident line $[\ell] \in \G\left( 1,r \right)$ with multiplicity $(d-k+1,1, \dots, 1)$.  
\end{proposition}
\begin{proof}
  The hypersurface $X$ does not have any $k'$-incident lines with $k' < k$.  Therefore, the proposition follows from point $2$ of \autoref{proposition:limitedorigin}.  
\end{proof}

\begin{proposition}
  \label{proposition:imagedandelion}
Suppose $k = d-(2r-2) \geq 3$. Then there exists a $k$-dandelion in the image of $\alpha$. 
\end{proposition}
\begin{proof}
  Consider the one-parameter family of polynomials $(z-t)(z-t^{2})\dots(z-t^{d-k})(z+t+t^{2}+\dots+t^{d-k})$ viewed as a map $\A^{1} \to \Delta_{d-k+1}$ given by $(r_{1}, \dots ,r_{d-k+1}) = (t,t^{2}, \dots, t^{d-k}, -t-t^{2}- \dots- t^{d-k})$.  If $\ell$ is a $k$-incident line of $X$ with multiplicities $(d-k+1,1,\dots,1)$ then $\Delta_{d-k+1}$ is a local chart  for any point $\hat{[\ell]} \in \Sigma^{[d]}$ lying above $\ell$, and therefore the one-parameter family above is realized, \'etale locally, as a curve on $\Sigma^{[d]}$.  Applying stable reduction to this one parameter family yields a $k$-dandelion, as the reader can easily verify. 
\end{proof}

We have all ingredients to prove the $m=1$ case of \autoref{theorem:main}.

\begin{theorem}
  \label{theorem:genericinjectivity}
  Suppose $d \geq 2r+2$, and $X \subset \P^{r}$ is a general degree $d$ hypersurface. Then the moduli map 
  \begin{align*}
    \mu_{1} : \G\left( 1,r \right) \dashrightarrow \Hyp(d,1)
  \end{align*}
  is generically injective.
\end{theorem}
\begin{proof}
  By \autoref{proposition:imagedandelion}, we know that there exists at least one $k$-dandelion $[P] \in \alpha(Z)$.

  If $\left[ P \right] \in \alpha(Z)$ is a $k$-dandelion, we have shown: The preimage $\alpha^{-1}([P])$ originates from  a unique $k$-incident line $\ell$, read off from the cross-ratios of the flower (\autoref{corollary:distinct-cross-ratios}), and the marking on the flower and stem of $[P]$ determines a point $\hat{[\ell]}$ and a point $z \in \beta^{-1}(\hat{[\ell]})$, respectively, showing that $\alpha^{-1}([P]) = z$. 
  
  We conclude from \autoref{proposition:wheredandelions} that the preimage of $[P]$ is the reduced point $z$, and from this the theorem follows.

\end{proof}

We now give the proof of \autoref{theorem:main}. 
\begin{proof}[Proof of \autoref{theorem:main}] 
By \autoref{theorem:genericinjectivity}, we know that the map $\mu_{1} : \G\left( 1,r \right) \dashrightarrow \Hyp(d,1)$ is generically injective. 

  Let $\Lambda \subset \P^{r}$ be a general $m$-plane.  The injectivity of $\mu_{1}$ implies that a general line $\ell \subset \Lambda$ will be general in the sense that there will be no other line $\ell'$ with $[X \cap \ell \subset \ell] \simeq [X \cap \ell' \subset \ell']$. 

  Therefore, if $\Lambda'$ is another $m$-plane with $[X \cap \Lambda \subset \Lambda] \simeq [X \cap \Lambda' \subset \Lambda']$, it must be that $\ell \subset \Lambda \cap \Lambda'$. Since this is true for a general $\ell \subset \Lambda$, we conclude that $\Lambda = \Lambda'$, which proves the theorem. 

\end{proof}

\subsection{Straight trees and the proof of \autoref{theorem:main2}}
\begin{notation}
  During this section, we assume $d = 2r+1$ and $X \subset \P^{r}$ a general hypersurface of degree $d$. 
\end{notation}

\begin{definition}
  \label{definition:straighttree}
  A {\sl straight tree} is any stable curve $(P,s_{1}, \dots, s_{d})$ having dual graph of the form 
  \tikzset{main node/.style={thick,circle,draw},}
  
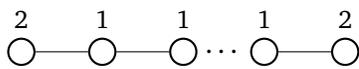
  \begin{tikzpicture}
    \node[main node, label={\small $2$}] (1) {};
    \node[main node, label={\small $1$}] (2) [right = .7cm of 1] {};
	\node[main node, label={\small $1$}] (3) [right = .7cm of 2] {};

    \node[main node, label={\small $1$}] (4) [right = .7cm of 3] {};
    \node[main node, label={\small $2$}] (5) [right = .7cm of 4] {};

    \path (1) edge (2);
    \path (2) edge (3);
    \path (3)  -- node[auto=false]{\ldots} (4);
    \path (4) edge (5);
  \end{tikzpicture}

(There are a total of $d-2$ vertices.)
\end{definition}

\begin{lemma}
  \label{lemma:numberstraighttrees}
  Assume $d \geq 4$. There are $d!/8$ distinct straight trees.
\end{lemma}
\begin{proof}
  Indeed, we may exchange the two points on each extremal component of a straght tree, or we may flip the entire tree over, without altering the moduli of the straight tree. Therefore, we must divide $d!$ by $2 \times 2 \times 2$. 
\end{proof}

\begin{lemma}
  \label{lemma:wherestraighttrees}
  Suppose $z \in Z$ is such that $\alpha(z) \in \barm_{0,d}$ is a straight tree. Then $\hat{[\ell]} := \beta(z) \in \Sigma^{[d]}$ lies over a $3$-incident line $\ell$. Furthermore, the multiplicity of $\ell$ is $(1,a,b)$, where $a+b = d-1$.
\end{lemma}
\begin{proof}
  $X$ possesses no $2$-incident lines.  The lemma follows from \autoref{proposition:limitedorigin}, and from the observation in \autoref{remark:stablereductionprocess}.
\end{proof}

\begin{lemma}
  \label{lemma:existsstraighttrees}
  Let $a,b$ be positive integers such that $a+b=d-1$. Then for every $3$-incident line $\ell$ of multiplicity $(1,a,b)$, and every point $\hat{[\ell]} \in \Sigma^{[d]}$ lying over $\ell$, there are precisely $a!b!/4$ points $z \in \beta^{-1}(\hat{[\ell]})$ such that $\alpha(z)$ is a straight tree. 
\end{lemma}
\begin{proof}
  Let $p, q$ be the points of multiplicity $a$ and $b$ on $\ell$. Then the variety $\beta^{-1}(\hat{[\ell]})$ is the collection of stable curves obtained by attaching all rational tails with $a$ marked points to the point $p$, and similarly with $q$. The lemma follows, after recalling that the two extremal points in a straight tree may be exchanged without changing the moduli of the tree -- this explains the division by $4$.  
\end{proof}

\begin{lemma}
  \label{lemma:11d-2}
  For every $3$-incident line $\ell$ of multiplicity $(1,1,d-2)$, and every point $\hat{[\ell]} \in \Sigma^{[d]}$ lying over $\ell$, there are precisely $(d-2)!/2$ points $z \in \beta^{-1}(\hat{[\ell]})$ such that $\alpha(z)$ is a straight tree.

\end{lemma}
\begin{proof}
The proof is completely analogous to the proof of the previous lemma. The only difference is that we are only attaching rational tails to {\sl one} point on $\ell$, hence we must divide by two only once.

  \end{proof}

  \begin{theorem}
    \label{theorem:enumerativtheoremmain}
    For any multiplicity $(a,b,c)$, let $n_{a,b,c}$ denote the number of $3$-incident lines to $X$ having the given multiplicity. 
    Then the degree of the map $\alpha: Z \to \barm_{0,d}$ is
    \begin{align*}
      \deg \alpha = 2 \sum_{a \geq b > 1}^{}n_{(a,b,1)} + 4n_{d-2,1,1}.
    \end{align*}
   
    Furthermore, $\deg \alpha = \deg \mu_{1}$. 
  \end{theorem}
  \begin{proof}
    The last statement is clear. 
    
    By symmetry considerations, every straight tree $[P] \in \barm_{0,d}$ has the same number of preimages in $Z$, all preimages arise from $3$-incident lines, by \autoref{lemma:wherestraighttrees}, and all preimages are unramified points of $\alpha$, by \autoref{corollary:localisomorphismontoimage}. 
    Above a $3$-incident line $\ell$ with multiplicity $(a,b,1)$, where $a,b>1$, there are $\frac{d!}{a!b!}$ points in $s^{-1}([\ell])$, by \autoref{lemma:numberpreimages}. Each contributes $a!b!/4$ straight trees by \autoref{lemma:existsstraighttrees}. Therefore, we obtain $\frac{d!}{a!b!}\frac{a!b!}{4} = \frac{d!}{4}$ total straight trees arising from $3$-incident lines with multiplicities $(a,b,1)$, where $a,b > 1$.  Since there are $\frac{d!}{8}$ straight trees total (\autoref{lemma:numberstraighttrees}), we see that each such line contributes $2$ to the degree of $\alpha$. 
    The same argument, using \autoref{lemma:11d-2} explains the coefficient $4$ for lines of type $(d-2,1,1)$. 
  \end{proof}

\section{Further questions} There are many remaining questions worth understanding. 

For instance, there is the question of extending \autoref{theorem:monodromymain} to the {\sl equi-multiplicity} setting, i.e. where the multiplicity vector is $(m,m,\dots,m)$ for some integer $m$.  Recall the situation for plane curves:  The monodromy groups of flexes of a plane cubic, and of bitangents of a plane quartic are not the respective full symmetric groups.  Does this special phenomenon persist in higher dimensions? The first open case is to determine the monodromy group of $5$-tangent lines of a general quintic surface in $\P^{3}$. 

\begin{remark}
  \label{remark:error}
  There is one instance in the literature worth noting. In \cite{dsouza}, it is argued that the monodromy group of the set of $4$-incident lines of type  $(2,2,2,2)$  to an octic surface $X \subset \P^{3}$ is the full symmetric group.  Unfortunately, it appears (to the author) that there is a small gap in the proof -- in the proof of Lemma 2.8, a particular variety $I_{8}$ is claimed to be locally irreducible, {\sl because it is irreducible}. This is false as it stands. 

  However, the author believes the result is still true, and that D'Souza's argument can be made to work. 
  \end{remark}

We have settled the enumerative question $(3)$ raised in the introduction in the case $m=1$ of lines.  However, there are many more instances where the preconditions of question $(3)$ are met. We give the following interesting open example: 
\begin{example}
  \label{example:quarticthreefold}
  Let $X_{4} \subset \P^{4}$ be a general quartic threefold.  What is the degree of the rational map 
  \begin{align*}
    \mu_{2}: \G\left( 2,4 \right) \dashrightarrow \m_{3}?
  \end{align*}
\end{example}
 
\begin{remark}
  \label{remark:approach}
  One possible approach, similar to that pursued in this paper, would be to count the number of ``complete quadrilaterals'' on $X_{4}$ -- these are plane curves which are the union of four general lines. The stable reduction procedure would also have to be analyzed. 
\end{remark}

Other enumerative puzzles emerge. 

\begin{example}
  \label{example:planesextic}
  Let $X_{6} \subset \P^{2}$ be a general plane sextic. Then the moduli map $\mu_{1}: \P^{2*} \dashrightarrow \m_{0,6}/S_{6}$ is generically injective.  Since the target is three dimensional, we expect the double-point locus of $\mu_{1}$ to be a curve $Y \subset \P^{2*}$.  Is this the case? 

  If so, the assignment $X \mapsto Y$ is a contra-invariant of ternary sextics -- what are the order and degree of this contra-invariant?
\end{example}

\begin{example}
  \label{example:Fermat}
  Cadman and Laza \cite{cadmanlaza} show that the degree of $\mu_{1}$ for a plane Fermat quintic is $150$. What is the degree in higher dimension, e.g. for the Fermat septimic surface in $\P^{3}$? We suspect that the answer is $7^{3} \times 24$, the size of the automorphism group of the Fermat septimic.  
\end{example}

Finally, and perhaps most importantly, there remains the question of improving \autoref{theorem:versal}. Indeed, it should be the case that the point-line correspondence is versal at every $k$-incident line with $k \geq d - (2r-2)$ for a general hypersurface $X$.  In fact, there should be an analogous versality result for higher dimensional plane sections of $X$ -- this is the subject of future work.

\bibliography{references}
\bibliographystyle{alpha}

\end{document}